\documentclass{amsart}

\usepackage{amsthm}
\usepackage{amssymb}
\usepackage{amsmath}
\usepackage{amsmath,amscd}
\usepackage{epsfig}

\newtheorem{theorem}{Theorem}
\newtheorem{corollary}{Corollary}
\newtheorem{lemma}{Lemma}

\newtheorem{definition}{Definition}

\newcommand{\er}[0]{\mathbf e_r}
\newcommand{\erp}[0]{\mathbf e_r'}
\newcommand{\ez}[0]{\mathbf e_z}
\newcommand{\ex}[0]{\mathbf e_x}
\newcommand{\ey}[0]{\mathbf e_y}

\title[Quadratic curvature blowup]{A minimal lamination of the interior of a positive cone with quadratic curvature blowup}

\author{Christine Breiner}
\address{Department of Mathematics, Fordham University, Bronx, NY 10458}
\email{cbreiner@fordham.edu}

\author{Stephen J. Kleene}
\address{Department of Mathematics, MIT, Cambridge, MA 02139}
\email{skleene@math.mit.edu}
\thanks{C. Breiner was supported in part by NSF grant DMS-1308420 and an AMS-Simons Travel Grant. S.J. Kleene was partially supported by NSF grant DMS-1004646.}
\begin{document}
\maketitle
\section{Introduction}
In this note we use elementary methods to construct a minimal lamination of the interior of a positive cone in $\mathbb R^3$. More precisely, we consider the immersion 
\begin{align} \label{delta_map} 
G(s, \theta) =    \left(e^{ \delta \theta} \sinh(s) \sin(\theta),  e^{ \delta \theta} \sinh(s) \cos (\theta),  \frac{1 }{ \delta} e^{ \delta \theta}   \right)
\end{align}
and show that a small graph over its image is a minimal surface for $s $ roughly proportional to $\delta^{-1 /4}$. The immersion $G(s, \theta)$ given above is a complete embedded $\infty$-valued disk on one side of plane, spiraling from above with quadratic curvature blowup, and the minimal graph over $G$ we find has boundary outside of a positive cone. If $h$ denotes the height above the plane and $|A|^2$ the second fundamental form of the surface, then
\begin{align} \label{lin_blow_up}
\sup_{G \cap \{ h > h_0 \}}| A |^2   \approx \delta^{-2}h_0^{-2}. 
\end{align}
We note that as $\delta$ tends to $0$, a suitable renormalization of $G$ converges to a  conformal parametrization of the helicoid. 

One cannot expect to find an embedded minimal graph over all of $G$ as such an example would contradict the result of Colding-Minicozzi concerning the properness of embedded minimal disks \cite{MMGPD,CY}. Indeed, using a blow up argument, one observes that an embedded minimal disk contained in a half space with quadratic curvature blow up as in (\ref{lin_blow_up}) cannot be extended beyond a positive cone with vertex at the blow up point. 

Colding and Minicozzi, using the Weierstrass Representation, provide an example of a minimal disk, with boundary on $\partial B_1$, spiraling into a plane with quartic curvature blow up \cite{CMPVNP}. Their example gives a lamination of $B_1\backslash \{0\}$ that does not extend as a lamination to all of $B_1$. In \cite{MPR}, Meeks, Perez, and Ros demonstrate that such a lamination can only occur with a faster than quadratic rate of blow up. It remains to show whether or not such laminations exist for rates between quadratic and quartic. For other, more pathological examples of laminations of open regions of $B_1$ by embedded minimal disks see \cite{BDeanPaper,HW3,SidPaper,Kl}.

Aside from studying lamination theory, one of the primary motivation for recording this result is the fact that it uses entirely elementary methods which seem to be widely applicable to related problems.

\section{The geometry of the initial immersion} 
In this section, we record relevant  geometric information about the immersion $G$. These include the unit normal, metric, second fundamental form and mean curvature of $G$. It will convenient to use the orthonormal basis $\{\er,\er',\ez\}$ in $\mathbb R^3$ to describe certain quantities on $G$ where $\er = (\sin \theta, \cos \theta, 0)$, $\er' = \partial_\theta \er$, and $\ez = (0,0,1)$. We note the components of the gradient and hessian of $G$ below:
\begin{align} \label{G_derivatives}
& G_s = e^{\delta \theta} \cosh(s) \er , \quad G_\theta = \delta e^{\delta \theta} \sinh(s)\er + e ^{\delta \theta} \sinh(s) \erp  +  e^{\delta \theta} \ez \\ \notag
& G_{s s} =   e^{\delta  \theta} \sinh(s) \er , \quad G_{s  \theta} = \delta e^{ \delta \theta} \cosh(s)\er + e^{\delta \theta} \cosh(s) \erp, \\ \notag
&  G_{\theta  \theta} =  \delta^2 e^{\delta \theta} \sinh(s) \er  - e^{\delta \theta} \sinh(s) \er + 2 \delta e^{\delta \theta} \sinh(s) \erp + \delta e^{\delta \theta} \ez.
\end{align}

\subsection{The unit normal}

\begin{lemma} \label{unit_normal}
The unit normal of $G$ is 
\begin{align}\label{normal}
\nu (s, \theta) = - \cosh^{-1}(s)\erp  + \tanh(s) \ez
\end{align}
\end{lemma}
\begin{proof}
This follows immediately as
\begin{align*}
G_s \wedge G_\theta & =  (e^{\delta \theta} \cosh(s) \er) \wedge( \delta e^{\delta \theta} \sinh(s)\er + e ^{\delta \theta} \sinh(s) \erp  +  e^{\delta \theta} \ez ) \\ \notag
& = e^{2\delta \theta}\cosh(s) \sinh(s) \er \wedge \erp +e^{2\delta \theta} \cosh(s) \er \wedge \ez\\ \notag
&  =e^{2\delta \theta} \cosh(s) \sinh(s)\ez -e^{2 \delta \theta} \cosh(s) \erp
\end{align*}
and thus
\begin{align*}
|G_s \wedge G_\theta|^2 = e^{4 \delta \theta} \cosh^4(s).
\end{align*}
\end{proof}

\subsection{The metric}

\begin{lemma} \label{metric}
Let $g = g _{s  s} ds^2  + g_{\theta  \theta} d\theta^2 + 2 g_{s  \theta} ds \, d \theta$ be the metric of $G$. Then
\begin{align*}
&g_{s s} = e^{2 \delta \theta} \cosh^2(s), \quad g_{\theta  \theta} = e^{2 \delta\theta} \left( \cosh^2(s) + \delta^2 \sinh^2(s)\right), \quad  g_{s \theta} = \delta e^{2 \delta \theta} \sinh(s) \cosh(s).
\end{align*}
\end{lemma}
\begin{proof}
This follows directly from (\ref{G_derivatives}).
\end{proof}

As a direct consequence, 
\begin{align} \label{determinant}
|g| : = \det g =  e^{4 \delta \theta} \cosh^4(s)
\end{align}
and the components of the dual metric are
\begin{align} \label{dual_metric}
g^{s  s} = e^{ - 2 \delta \theta}\cosh^{-2 }(s) (1 + \delta^2 \tanh^2(s)), \quad g^{\theta \theta} = e^{- 2 \delta \theta} \cosh^{-2}(s), \quad g^{s  \theta} = - \delta e^{ - 2 \delta \theta} \tanh(s) \cosh^{-2} (s).
\end{align}

\subsection{The second fundamental form}

\begin{lemma} \label{2FF}
Let $A :  = A_{s  s} ds^2 + A_{\theta  \theta} d \theta^2 + 2 A_{s  \theta} ds\,  d \theta$ be the second fundamental form, and let $|A|^2$ be its length. Then we have
\begin{align*}
A_{s  s} = 0, \quad A_{\theta \theta} = - \delta e^{\delta \theta} \tanh(s), \quad A_{s \theta} = - e^{\delta \theta}
\end{align*}
and 
\begin{align*}
|A|^2 = e^{- 2 \delta \theta} \cosh^{ - 4} (s)(2 + 2\delta^2 \tanh^2(s)).
\end{align*}
\end{lemma}
\begin{proof}
We determine the components of the second fundamental form by using (\ref{G_derivatives}) and \eqref{normal}. To obtain the expression for the length of the second fundamental form, we write
\begin{align*}
|A|^2 &= 
\left(\begin{array}{ccc}
A_{s  \theta} A_{s  \theta} & A_{s  \theta} A_{\theta  s} & A_{s  \theta} A_{\theta  \theta} \\
& \\
A_{\theta  s} A_{s  \theta } &  A_{ \theta  s} A_{\theta  s} &A_{\theta s} A_{\theta \theta} \\
& \\
A_{\theta  \theta} A_{s  \theta} & A_{\theta \theta} A_{\theta  s} & A_{\theta  \theta} A_{\theta  \theta}
\end{array} \right) * 
 \left(\begin{array}{ccc}
g^{s  s} g^{\theta  \theta} & g^{s  \theta} g^{\theta  s} & g^{s  \theta} g^{\theta  \theta}\\
& \\
g^{\theta  s} g^{s  \theta}&  g^{\theta  \theta} g^{s  s} & g^{\theta  \theta} g^{s  \theta}\\
& \\
g^{\theta  s} g^{\theta  \theta} & g^{\theta  \theta} g^{\theta  s} & g^{\theta \theta} g^{\theta  \theta}
\end{array} \right) \\ \notag
\\
& = 
e^{2\delta \theta}
\left(\begin{array}{ccc}
1 & 1 &  \delta  \tanh(s)\\ \notag
 \\
1 & 1 &  \delta \tanh(s)\\ \notag
\\
 \delta \tanh(s) &  \delta \tanh(s) & \delta^2 \tanh^2(s) 
\end{array} \right)
*
\frac{e^{- 4\delta \theta}}{\cosh^{ 4} (s) }
\left(\begin{array}{ccc}
1 + \delta^2 \tanh^2(s) & \delta^2 \tanh^2 (s) & - \delta \tanh(s) \\
\\
\delta^2 \tanh^2 (s) & 1 + \delta^2 \tanh^2(s) &  - \delta \tanh(s) \\
\\
-\delta \tanh(s) & - \delta \tanh(s) & 1
\end{array} \right) \\ \notag
\\
& = \cosh^{-4 }( s) e^{- 2\delta \theta} \left( 2 + 2\delta^2 \tanh^2(s)\right).
\end{align*}
\end{proof}

\begin{corollary} \label{mean_curvature}
Let $H$ be the mean curvature of $G$. Then 
\begin{align*}
H (s, \theta) = - \delta e^{- \delta \theta} \tanh(s) \cosh^{-2} (s).
\end{align*}
\end{corollary}

\subsection{The Laplace operator}
\begin{lemma} \label{laplace_operator}
Let $\Delta_g$ be the Laplace operator on $G$. Then 
\begin{align*}
e^{ 2 \delta \theta} \cosh^{2} (s)\Delta_{g} & = (1 + \delta^2\tanh^2(s)) \partial_{s  s} + \partial_{\theta  \theta} - 2 \delta \tanh(s) \partial_{s  \theta}  \\ \notag
&\quad +  2 \delta^2 \tanh(s)\cosh^{ - 2}(s) \partial_s  - \delta \cosh^{-2} (s) \partial_\theta.
\end{align*}
\end{lemma}
\begin{proof}
This follows directly from the expressions for the coefficients of the dual metric and its determinant in (\ref{determinant}) and (\ref{dual_metric}).
\end{proof}
\section{Correcting the mean curvature}
We seek a graph over $G$ by a function $w(s, \theta)$ of the form 
\begin{align} \label{w_ansatz}
w(s, \theta) =  e^{\delta \theta} u (s),
\end{align}
so that  $H_w  \equiv 0$. Here we have denoted by $H_w$ the mean curvature of $G_w$, the normal graph over $G$ by $w$:
\begin{align*}
G_w (s, \theta) : = G(s, \theta) + w (s, \theta)\nu (s, \theta).
\end{align*}
In order to formulate the problem conveniently, we discuss some general properties relating to the mean curvature of immersions. 

\subsection{Mean curvature of immersions}

The mean curvature of an immersion $\phi(s, \theta)$ is computed by a homogeneous degree $-1$ function $H = H (\underline{\nabla})$,  defined on the euclidean space $E = \mathbb R^{3 \times 2} \times\mathbb R^{3 \times 3 \times 3}$. We denote points of $E$ by $\underline{\nabla} : = (\nabla, \nabla^2)$, where
\begin{align*}
\nabla = (\nabla_s, \nabla_\theta), \quad \nabla^2  = (\nabla_{s  s}, \nabla_{\theta \theta}, \nabla_{s \theta}).
\end{align*}
We can then explicitly write
\begin{align} \label{mean_curvature_function}
H (\underline{\nabla}) = \left\{(\nabla^T \nabla) *( \nabla^2 \cdot \nabla_s \wedge \nabla_\theta) \right\} / \left(\det{\nabla^T \nabla}\right)^{3/2}.
\end{align}
$H (\underline{\nabla})$ is then  homogeneous degree $-1$ in the sense that 
\begin{align*}
H (c \underline{\nabla})  = c^{-1} H (\underline{\nabla})
\end{align*}
when defined. This immediately gives that  the $j^{th}$ derivative of $H$ is homogeneous degree $-1 - j$ in the sense that
\begin{align*}
\left. D^{(j)}H \right|_{ c \underline{\nabla}} = c^{-1 - j}   \left. D^{(j)}H \right|_{  \underline{\nabla}}.
\end{align*}
Let $R$ be a rotation of $R^3$, and let $R\underline{\nabla}$ be the point of $E$ obtained by rotating the column vectors of $\underline{\nabla}$ by $R$. Then $H$ is  invariant under rotations  in the sense that 
\begin{align*}
H (R \underline{\nabla}) = H(\underline{\nabla}).
\end{align*}
For a point $\underline{\nabla} $ in $E$, we set
\begin{align*}
\mathfrak{a} (\underline{\nabla})  = \mathfrak{a} (\nabla) : = \left( \frac{\nabla_s}{ |\nabla_s |} \cdot \frac{\nabla_\theta}{| \nabla_\theta|}, \frac{|\nabla_s|}{|\nabla_\theta |} - 1\right).
\end{align*}
Note that $\mathfrak{a} (\nabla)$ vanishes exactly when the columns of $\nabla$ are orthogonal and of the same length. Then,  as long as $\left|\mathfrak{a} (\underline{\nabla})\right|$ is sufficiently small (say, less than $1/4$) it holds that 
\begin{align}\label{DH_est}
\left|\left. D^{(j)}  H \right|_{\underline{\nabla}} \right| < C  |\nabla|^{ - 1- j} \left( 1 + |\nabla^2|/ |\nabla|\right) 
\end{align}
since 
\begin{align*}
\left. D^{(j)}  H \right|_{\underline{\nabla}} = |\nabla |^{- 1 - j}\left. D^{(j)}  H \right|_{\underline{\nabla}/ |\nabla|}
\end{align*}
and $\det{\nabla^T \nabla}$ in (\ref{mean_curvature_function}) is uniformly smooth on  compact subsets of $E \setminus \{ 0 \}$. Given an immersion $\phi$, set 
\begin{align*}
\underline{\nabla} \phi : = (\nabla \phi, \nabla^2 \phi),
\end{align*}
so that the mean curvature of $\phi$ is given by
\begin{align*}
H_{\phi} : = H (\underline{\nabla} \phi).
\end{align*}
We also note that from ($\ref{G_derivatives}$) 
\begin{align} \label{G_almost_conformal}
\left|\mathfrak{a} (\underline{\nabla} G) \right|< C \delta.
\end{align}

\subsection{Separating variables}
Let $R_\theta$ be the  rotation 
\begin{align*}
R_\theta = \ex^* \otimes \er + \ey^* \otimes \erp + \ez^* \otimes \ez
\end{align*}
where $\{\ex^*, \ey^*, \ez^*\}$ is the dual basis in $\mathbb R^3$ to the standard basis. Set $ \tilde{\nabla}_0: = e^{- \delta \theta} R_\theta^{-1}\underline{\nabla} G $. For a function  $u(s)$ we abuse notation slightly and set
\begin{align*}
\underline{\nabla}_\delta u(s) : = e^{- \delta \theta} R_\theta^{-1} \underline{\nabla}( e^{\delta \theta} u \nu) := (\nabla_{\delta} u, \nabla^2_\delta u).
\end{align*}
By (\ref{G_derivatives}) and \eqref{unit_normal}, both $\tilde{\nabla}_0$ and $\underline{\nabla}_\delta u$ are independent of $\theta$. We can then write
\begin{align} \label{w_mean_curvature}
H_w & = H (\underline{\nabla} (G + \nu w) ) =  e^{- \delta \theta} H ( e^{- \delta \theta}\underline{\nabla}(G +w\nu )) \\ \notag
& =  e^{- \delta \theta}H (\tilde{\nabla}_0 + \underline{\nabla}_\delta u) \\ \notag
& : = e^{- \delta \theta} \cosh^{-2} (s)Q_\delta (u).
\end{align}
Note that $Q_\delta(u)$ is a fixed second order differential operator without dependence on $\theta$ since $\tilde{\nabla}_0$ and $\underline{\nabla}_\delta u$ depend only on $s$. The linearization of the operator  $Q_\delta$ is then
\begin{align}
  L_\delta u &: = e^{\delta \theta} \cosh^{2} (s) \mathcal{L} e^{\delta  \theta} u \\ \notag
 &  = \left( 1 + \delta^2 \tanh^2(s)\right) u'' + \delta^2 u - 2 \delta^2\tanh(s) u'   +2 \delta^2 \tanh (s) \cosh^{-2} (s) u'  \\ \notag
  &\qquad - \delta^2 \cosh^{-2} (s)u + 2 \cosh^{-2}(s) u + 2\delta^2 \tanh^2 (s) \cosh^{-2}(s) u \\ \notag
  & := L_0 (u) + O_\delta (u),
\end{align}
where above we have set
\begin{align}
L_0 (u) : = u'' (s)+ 2  \cosh^{ - 2}(s) u(s).
\end{align}
For every $f \in C^{0,\alpha}(\mathbb R)$, let
\begin{align}
L_0^{-1} (f) (s): = \left(\int_0^s \tanh^{-2} (s') \int_0^{s'} \tanh(s'') f(s'') d s'' ds' \right) \tanh(s).
\end{align}
Then
\begin{lemma}
$L_0^{-1}$ is a right sided inverse for $L_0$.
\end{lemma}
\begin{proof}
The expression can be found by trying the ansatz $L_0 (\tanh(s) u) = f$. Since $L_0 ( \tanh(s)) = 0$, this gives the reduced order equation 
\begin{align*}
\tanh (s) u'' + 2 \tanh' (s) u ' = f.
\end{align*}
Integrating then gives the expression for $L_0^{-1}$.
\end{proof}

\subsection{Finding an exact solution}
We construct the solution $u(s)$ to $Q_\delta(u) = 0$ as  a fixed point for the map
\begin{align}
\Psi (u) : = u - L_0^{-1} Q_\delta (u).
\end{align}
As usual we do this by appealing to a fixed point theorem. For this, we need a few estimates and appropriate norms.
\begin{definition}
For $u:\Omega \to \mathbb R$ we define its local H\"older function by
\[
\|u\|_{k,\alpha}(s):= {\|u:C^{k,\alpha}(B_1(s) \cap \Omega)\|}.
\]
\end{definition}

\begin{lemma} \label{delta_mc_perturbation_estimate}
Assume that 
\begin{align} \label{u_generous_smallness_assumption}
\| u \|_{ {2,\alpha}} (s) < \varepsilon \cosh (s).
\end{align}
Then for $\varepsilon$ sufficiently small there exists $C_1$ independent of $\varepsilon$ such that
\begin{align*}
\|Q_\delta (u) - Q_0 (u) \|_{{0,\alpha}} (s)   < C_1 \delta \| u\|_{{2,\alpha}} (s). 
\end{align*}
\end{lemma}
\begin{proof}
Write
\begin{align*}
Q_\delta (u) - Q_0 (u) & = \cosh^{2} (s) H (\tilde{\nabla}_0 + \underline{\nabla}_\delta u) - \cosh^{2} (s) H (\tilde{\nabla}_0 + \underline{\nabla}_0 u) \\ \notag
& = \cosh^2 (s) \int_0^1 \left. D H \right|_{\underline{\nabla} (\sigma)} \left( \underline{\nabla}_\delta u - \underline{\nabla}_0 u \right) d \sigma
\end{align*}
where we have set $\underline{\nabla} (\sigma) : = \tilde{\nabla}_0 + \sigma  \underline{\nabla}_\delta u + (1 - \sigma )  \underline{\nabla}_0 u $. By (\ref{G_almost_conformal}) and (\ref{u_generous_smallness_assumption}), there exists a large $C$ independent of $\delta, \varepsilon$ such that
\begin{align} \label{good_nabla_bounds}
C^{-1} \cosh (s) < |\nabla (\sigma)| < C \cosh(s);\quad \quad | \nabla^2 (\sigma)| < C \cosh(s);  \quad \quad \left| \mathfrak{a} (\nabla (\sigma)) \right| < C\left( \varepsilon  + \delta \right).
\end{align}
Since $D H$ is homogeneous degree $-2$, \eqref{DH_est} implies
\begin{align*}
\left\| Q_\delta (u) - Q_0 (u) \right\|_{{0,\alpha}}(s)   & <  C \| \underline{\nabla}_\delta u  - \underline{\nabla}_0u\|_{{0,\alpha}}(s) \sup_{\sigma}\left\{ \frac{| \underline{\nabla} (\sigma)|}{ \cosh}(s) \right\}\\ \notag
& < C\delta \| u \|_{{2,\alpha}}(s) .
\end{align*}The final inequality follows from the uniform bounds on the derivatives of $\nu$ and the definition of $\underline \nabla_\delta u, \underline \nabla_0 u$.
\end{proof}

\begin{lemma} \label{quadratic_remainder_estimate}
Assume that $u(s), v(s)$ satisfy (\ref{u_generous_smallness_assumption}). Then there exists $C_2$ independent of $\varepsilon$ such that
\begin{align*}
\left\| Q_\delta (v) - Q_\delta (u) - L_\delta (v - u) \right\|_{{0,\alpha}}(s)  < C_2 \cosh^{-1} (s) \| v - u \|^2_{{2,\alpha}}(s). 
\end{align*}
\end{lemma}
\begin{proof}
We can write
\begin{align*}
 Q_\delta (v) - Q_\delta (u) - L_\delta (v - u)  = \cosh^{2} (s) \int_0^{1}\int_0^\sigma \left. D^{(2)} H \right|_{\underline{\nabla} (\sigma')} (\underline{\nabla}_0 (v - u), \underline{\nabla}_0 (v - u))d\sigma' \,d\sigma
\end{align*}
where we have set $\underline{\nabla} (\sigma') = \tilde{\nabla}_0 + \sigma' \underline{\nabla}_0 v + (1 - \sigma')  \underline{\nabla}_0 u$. Again, the assumptions give that $\underline{\nabla} (\sigma')$ satisfies the estimates in (\ref{good_nabla_bounds}), so that we get the estimate
\begin{align*}
 \left\|Q_\delta (v) - Q_\delta (u) - L_\delta (v - u) \right\|_{{0,\alpha} }(s) & < C \| u - v \|^2 _{{2,\alpha}}(s) \cosh^{-1}(s) \sup_{\sigma} \left\{ \frac{|\underline{\nabla} (\sigma')|}{ \cosh}(s)\right\}.
\end{align*}
\end{proof}

The function spaces in which we find the exact solution are the spaces $\mathcal{X}_k$ of functions on $[-\varepsilon/\delta^{1/4},  \varepsilon/\delta^{1/4}]$ equipped with the norm
\begin{align*}
\| f\|_{\mathcal{X}, k} :  =  \sup_{|s| \leq \varepsilon/\delta^{1/4}}\frac{\| f : C^{k, \alpha}(B_1(s) \cap [-\varepsilon/\delta^{1/4},  \varepsilon/\delta^{1/4}]) \|}{|s|^k}.
\end{align*}
We note that 
\begin{lemma} \label{linear_operator_inversion}
$L_0^{-1} : \mathcal{X}_0 \rightarrow \mathcal{X}_2$  is a bounded linear map with bound $C_3>0$.
\end{lemma}
\begin{proof}
Let $f \in \mathcal{X}_0$ and write $A : = \| f\|_{\mathcal{X}, 0}$. Then
\begin{align*}
\left| L_0^{-1} (f) \right| (s) &  < A\left(\int_0^s \tanh^{-2} (s') \int_0^{s'} \tanh(s'') d s'' ds' \right) \tanh(s) \\ \notag
& < C_3 A s^2.
\end{align*}
The estimate for the derivatives of $L_0^{-1} (f)$ follow similarly, so that we have $\| L_0^{-1} (f)\|_{\mathcal{X},2} < C_3\| f\|_{\mathcal{X}, 0}$ as claimed. 
\end{proof}
We also note the important fact that
\begin{lemma}
On the space of functions $\mathcal X_2$, $L_0^{-1}$ is a left inverse of $L_0$.
\end{lemma}
\begin{proof}
Let $u \in \mathcal X_2$ and $L_0(u)=f \in \mathcal X_0$, and suppose $L_0^{-1}(f)=w$. By the definition of $L_0^{-1}$, $L_0(w)=f$ and $w(0)=w'(0)=0$. By standard ODE theory, since $L_0(u-w)=0$ and $(u-w)(0)=(u-w)'(0)=0$, we conclude $u=w$.\end{proof}
We set
\begin{align} \label{schauder_ball}
\mathcal{B} : = \left\{ f \in \mathcal{X}_2: \| f\|_{\mathcal{X}, 2} \leq  \zeta \delta \right\}.
\end{align}

The main result is then
\begin{theorem} \label{main_theorem}
For $\zeta$ sufficiently large in \eqref{schauder_ball} there exists $\delta_0$ such that for all $0<\delta< \delta_0$, there exists $u \in \mathcal{B}$ such that $Q_\delta(u) (s) = 0$. 

\end{theorem}
\begin{proof}Set $C_4:= \max\{C_1, C_2, C_3\}$ where the constants come from Lemmas \ref{delta_mc_perturbation_estimate}, \ref{quadratic_remainder_estimate}, and  \ref{linear_operator_inversion}. Choose $\zeta>1$ such that $C_4/\zeta <1/4$. Choose $0<\varepsilon \leq \min\{1, \frac{1}{4C_4(1+\zeta)}\}$. For this $\zeta, \varepsilon$, choose $0<\delta_0< \min\{1, \frac{\varepsilon}{4\zeta}\}$.

Set 
\begin{align*}
\Psi (u)  =u -L_0^{-1}Q_\delta (u).
\end{align*}
Then for $u,v \in \mathcal B$
\begin{align*}
\Psi (v) - \Psi (u)&  = v - u - L_0^{-1} \left( Q_\delta (v) - Q_\delta(u) \right)  = L_0^{-1} \left( L_0 (v - u) - Q_\delta (v) + Q_\delta (u) \right) \\ \notag 
& =  L_0^{-1} \left( L_0 (v - u) - Q_0 (v) + Q_0 (u) \right)  \\ \notag
& \quad -  L_0^{-1} \left(Q _\delta(v)- Q_0 (v) \right) +  L_0^{-1} \left(Q _\delta(u)- Q_0 (u) \right) \\ \notag
& = I + II + III.
\end{align*}
Given the bound on $\delta$, one can apply Lemma \ref{delta_mc_perturbation_estimate}. Taken with Lemma \ref{linear_operator_inversion}, the estimates imply
\begin{align}
\left\| II \right\|_{\mathcal{X}, 2} \leq C_4\|  Q _\delta(v)- Q_0 (v) \|_{\mathcal{X},0} < C_4 \delta\|v\|_{{2,\alpha}}(s)\leq C_4\delta^2 \zeta s^2 \leq C_4\delta^{3/2}\zeta \varepsilon^2 < C_4\varepsilon \delta \zeta.
\end{align}
The same argument gives the estimate
\begin{align}
\left\| III \right\|_{\mathcal{X}, 2} < C_4 \varepsilon \zeta \delta.
\end{align}
To obtain the estimate for $I$, we use Lemma \ref{quadratic_remainder_estimate} to get
\begin{align*}
  \left\| I \right\|_{\mathcal{X}, 2}  <  C_4\left\| L_0 (v - u) - Q_0 (v) - Q_0 (u) \right\|_{\mathcal{X},0} <C_4  \sup_{s}\left(  \cosh^{-1} (s) \delta^2 \zeta^2  s^4 \right) < C_4 \varepsilon \delta \zeta^2.
 \end{align*}
Together the estimates imply
 \begin{align*}
 \left\| \Psi (v) - \Psi (u) \right\|_{\mathcal{X}, 2} < C_4\left(\varepsilon  + \varepsilon \zeta \right)\delta \zeta < \frac{\delta \zeta }{4}. 
 \end{align*}
Additionally,  since $Q_\delta (0) = \delta \tanh(s)$, we have from Lemma \ref{linear_operator_inversion} that 
 \begin{align*}
 \| \Psi (0)\|_{\mathcal{X}, 2} < C_4 \delta.
 \end{align*}
Therefore, the condition on $\zeta$ implies 
 \begin{align*}
 \| \Psi (v)\|_{\mathcal{X}, 2} < \frac{\delta \zeta}{4} + C_4 \delta  = \left( 1/4 + C_4/ \zeta\right) \delta \zeta < \delta \zeta .
 \end{align*}
Thus, $\Psi (\mathcal{B}) \subset \mathcal{B}$. Since $\mathcal{B}$ is a compact, convex subset of a Banach space, the Schauder fixed point theorem gives the existence of a function $u(s) \in \mathcal{B}$ such that $u = u - L_0^{-1}Q_\delta(u)$. Taking $L_0$ of both sides we determine $Q_\delta( u) \equiv 0$ as claimed. 
\end{proof}

\subsection{Embeddedness and minimality of the graph }
Let $w_\delta (s, \theta)$ be the function given by
\begin{align}
w_\delta (s, \theta) : = 
e^{\delta \theta} u(s).
\end{align}
By (\ref{w_mean_curvature}), the normal graph $G_{w_\delta}$ over $G$ by $w_{\delta}$ is an immersed minimal surface.  Since $| w_\delta (s, \theta)| <  e^{\delta \theta}\zeta \delta s^2 < e^{\delta \theta}  \varepsilon \delta^{1/2}/4$, the embeddedness of $G_{w_\delta}$ follows from the definition of the map $G$ in (\ref{delta_map}).

\bibliographystyle{amsplain}
\bibliography{Biblio}

\providecommand{\bysame}{\leavevmode\hbox to3em{\hrulefill}\thinspace}
\providecommand{\MR}{\relax\ifhmode\unskip\space\fi MR }
\providecommand{\MRhref}[2]{%
  \href{http://www.ams.org/mathscinet-getitem?mr=#1}{#2}
}
\providecommand{\href}[2]{#2}
\begin{thebibliography}{1}

\bibitem{MMGPD}
T.~H. Colding and W.~P.~Minicozzi II, \emph{Multivalued minimal graphs and
  properness of disks}, Int. Math. Res. Not. (2002), no.~21, 1111--1127.

\bibitem{CMPVNP}
\bysame, \emph{Embedded minimal disks: Proper versus nonproper - global versus
  local}, Trans. Amer. Math. Soc. \textbf{356} (2004), 283--289.

\bibitem{CY}
\bysame, \emph{{The Calabi-Yau conjectures for embedded surfaces}}, Ann. of
  Math. \textbf{167} (2008), no.~1, 211--243.

\bibitem{BDeanPaper}
B.~Dean, \emph{{Embedded minimal disks with prescribed curvature blowup}},
  Proc. Amer. Math. Soc. \textbf{134} (2006), no.~4, 1197--1204.

\bibitem{HW3}
D.~Hoffman and B.~White, \emph{Sequences of embedded minimal disks whose
  curvatures blow up on a prescribed subset of a line}, Comm. Anal. Geom.
  \textbf{19} (2011), no.~3, 487--502.

\bibitem{MPR}
W.~H.~Meeks III, J.~P\'erez, and A.~Ros, \emph{Local removable singularity
  theorems for minimal laminations}, preprint.

\bibitem{SidPaper}
S.~Khan, \emph{{A Minimal Lamination of the Unit Ball with Singularities along
  a Line Segment}}, Illinois J. Math \textbf{53} (2009), no.~3, 833--855.

\bibitem{Kl}
S.~Kleene, \emph{A minimal lamination with {C}antor set-like singularities},
  Proc. Amer. Math. Soc. \textbf{140} (2012), no.~4, 1423--1436.

\end{thebibliography}
\end{document}